\documentclass[11pt]{article}
\usepackage[utf8]{inputenc}
\usepackage[english]{babel}
\usepackage{amssymb,amsmath}
\usepackage[usenames, dvipsnames]{color}
\usepackage{amsthm}
\usepackage{cite}
\let\OLDthebibliography\thebibliography
\renewcommand\thebibliography[1]{
  \OLDthebibliography{#1}
  \setlength{\parskip}{0pt}
  \setlength{\itemsep}{0pt plus 0.3ex}
}
\newtheorem{theorem}{{\scshape Theorem}}[section]
\newtheorem{lemma}[theorem]{{\scshape Lemma	}}
\newtheorem{corollary}[theorem]{{\scshape Corollary}}

\begin{document}
\title{Finite skew braces with solvable additive group}
\author{Ilya Gorshkov, Timur Nasybullov}
\date{}
\maketitle
\begin{abstract}
A.~Smoktunowicz and L.~Vendramin conjectured that if $A$ is a finite skew brace with solvable additive group, then the multiplicative group of $A$ is solvable. In this short note we make a step towards positive solution of this conjecture proving that if $A$ is a minimal finite skew brace with solvable additive group and non-solvable multiplicative group, then the multiplicative group of $A$ is not simple. On the way to obtaining this result, we prove that the conjecture of A.~Smoktunowicz and L.~Vendramin is correct in the case when the order of $A$ is not divisible by $3$.\\

\end{abstract}
\section{Introduction}A skew brace $A=(A,\oplus,\odot)$ is an algebraic system with two binary algebraic operations  $\oplus$, $\odot$ such that $A_{\oplus}=(A,\oplus)$, $A_{\odot}=(A,\odot)$ are groups and the equality
\begin{equation}\label{mainleft}
a\odot(b\oplus c)=(a\odot b)\ominus a \oplus (a\odot c)
\end{equation}
holds for all $a,b,c\in A$, where $\ominus a$ denotes the inverse to $a$ element with respect to the operation~$\oplus$ (we denote by $a^{-1}$ the inverse to $a$ element with respect to $\odot$). The group $A_{\oplus}$ is called the additive group of a skew brace $A$, and the group $A_{\odot}$ is called the multiplicative group of a skew brace $A$. If $A_{\oplus}$ is abelian, then $A$ is called a classical brace.

Classical braces were introduced by Rump in \cite{Rum} in order to study non-degenerate involutive
set-theoretic solutions of the Yang-Baxter equation. Skew braces were introduced by Guarnieri and Vendramin in \cite{GuaVen} in order to study non-degenerate   
set-theoretic solutions of the Yang-Baxter equation which are not necessarily involutive. Skew braces have connections with other algebraic structures such as groups with exact factorizations, Zappa-Sz\'{e}p products,
triply factorized groups and Hopf-Galois extensions  \cite{SmoVen}. Some algebraic aspects of skew braces are studied in \cite{Chi, CedSmoVen, Jes, Nas, SmoVen, KonSmoVen}. Thanks to its relations with the Yang-Baxter equation skew braces can be applied for constructing representations of virtual braid groups and invariants of virtual links\cite{BarNasmult1, BarNasmult2}. A big list of problems concerning skew braces is collected in \cite{Ven}.

Equality (\ref{mainleft}) makes the additive group $A_{\oplus}$ and the multiplicative group $A_{\odot}$ of a skew brace $A$ strongly connected with each other. For example, the following problems formulated in the Kourovka notebook \cite[Problem 19.90]{kt} tell about some of such connections.
\begin{enumerate}
\item  Does there exist a (finite) skew brace with nilpotent multiplicative group but non-solvable additive group?
\item  Does there exist a (finite) skew brace with solvable additive group but non-solvable multiplicative group?
\end{enumerate}
The first problem from the list above was studied, for example, in \cite{Byo, SmoVen, pr2doo, Nas}. The most general result about this question is obtained in \cite{pr2doo}, where it is proved that the answer to this question is negative in the case of finite skew braces. The second problem from the list above was studied, for example, in \cite{SmoVen, Nas}. It is known \cite[Example 3.2]{Nas} that there exists an infinite skew brace which gives a positive answer to this question. So, in general, the answer to the second question from the list above is positive. In this short note we make a step towards negative answer to the second question from the list above for finite skew braces. More precisely, we prove the following theorem. 

~\\
\textbf{Theorem \ref{maintheorem}.} Let $A$ be a minimal finite skew brace, such that $A_{\oplus}$ is solvable, and $A_{\odot}$ is not solvable. Then $A_{\odot}$ is not simple.

~\\
On the way to obtaining this result, we prove the following statement which proves that the answer to the second question from the list above is negative for some finite skew braces.

~\\
\textbf{Corollary \ref{corma}.} Let $A$ be a finite skew brace such that the additive group $A_{\oplus}$ is solvable. If $|A|$ is not divisible by $3$, then $A_{\odot}$ is solvable.

~\\
\textbf{Acknowledgment.} The work is supported by Mathematical Center in Akademgorodok under agreement No. 075-15-2019-1613 with the Ministry of Science and Higher Education of the Russian Federation.

\section{Preliminaries}
In this section we formulate and prove the necessary preliminary results.

Let $A$ be a skew brace.  Equality (\ref{mainleft}) implies that the unit element of $A_{\oplus}$ coincides with the unit element of $A_{\odot}$, we denote this element by $1$. 
For $a\in A$ the map $\lambda_a: x\mapsto \ominus a\oplus (a\odot x)$ is an automorphism of the group $A_{\oplus}$. Moreover, the map $a\mapsto \lambda_a$ gives a homomorphism $A_{\odot}\to {\rm Aut}(A_{\oplus})$. From the definition of $\lambda_a$ follows that for all $a,b\in A$ the equality
\begin{equation}\label{multsum}
a\odot b=a\oplus \lambda_a(b)
\end{equation}
holds. From this equality follows, in particular, that for each $a\in A$ the element $a^{-1}$ has the following form
\begin{equation}\label{9onverse0}
a^{-1}=\ominus\lambda_a^{-1}(a).
\end{equation}
\begin{lemma}\label{characteristicsub}Let $G$ be a  finite group, and $p$ be a prime divisor of the index $|G:[G,G]|$. Then there exists a characteristic subgroup $H$ of $G$ such that $G/H=\mathbb{Z}_p^n$ for some $n>0$.
\end{lemma}
\begin{proof}Let $|G:[G,G]|=p^km$, where $k>0$ and $(p,m)=1$. Denote by $A$ the subset 
$$A=\{x\in G~|~x^m\in [G,G]\}$$
of $G$. It is clear that $A$ is a subgroup of $G$ such that $[G,G]$  is a subgroup of $A$, and $|G:A|=p^k$. Let $\varphi$ be an automorphism of $G$, and $x$ be an element from $A$.  Then $x^m\in[G,G]$, and therefore $\varphi(x)^m\in [G,G]$, i.~e. $\varphi(x)$ is an element from $A$, so, $A$ is a characteristic subgroup of $G$. Since $[G,G]\leq A$, the quotient $G/A$ is abelian $p$-group.

Denote by  $B=\{x^p~|~x\in G/A\}$ the subgroup of $G/A$. Since  $G/A$ is abelian $p$-group, it is clear that $B\neq G/A$. Let $H$ be the preimage of $B$ in $G$
$$H=\{x^pa~|~x\in G, h\in A\}.$$
It is clear that $H$ is a subgroup of $G$ such that $[G,G]\leq A\leq H$. Let $\varphi$ be an automorphism of $G$, and $x$ be an element from $H$.  Then $x=y^pa$ for some $y\in G$, $a\in A$, and therefore $\varphi(x)^m=\varphi(y)^p\varphi(a)$. Since $A$ is characteristic in $G$, the element $\varphi(x)$ belongs to $H$, so, $H$ is a characteristic subgroup of $G$. Since $[G,G]\leq A\leq H$, the quotient $G/H$ is a finite abelian group. Since for every $x\in G$ the element $x^p$ belongs to $H$, every element from $G/H$ is of order $p$, therefore, $G/H=\mathbb{Z}_p^n$. Since $B=H/A\neq G/A$, the index $|G:H|$ is equal to
$$|G:H|=|G:A|/|H:A|\neq 1,$$
therefore $n>0$.
\end{proof}
\begin{corollary}\label{corma}Let $A$ be a finite skew brace such that the additive group $A_{\oplus}$ is solvable. If $|A|$ is not divisible by $3$, then $A_{\odot}$ is solvable. 
\end{corollary}
\begin{proof}Suppose that the statement is not correct and let $(A,\oplus,\odot)$ be a counter example of minimal order. Since $A_{\oplus}$ is solvable, the index $|A_{\oplus}:[A_{\oplus},A_{\oplus}]|$ is divisible by some prime number $p$. From Lemma~\ref{characteristicsub} follows that in $A_{\oplus}$ there exists a characteristic subgroup $H$ of index $p^{n}$. Consider arbitrary elements $a,b\in H$. Since $\lambda_a$ is an automorphism of $A_{\oplus}$, and $H$ is a characteristic subgroup of $A_{\oplus}$, the element $\lambda_a(b)$ belongs to $H$, and the element $a\odot b=a\oplus\lambda_a(b)$ belongs to $H$. Therefore the set $H$ is a subgroup of $A_{\odot}$, and hence $H$ is a skew subbrace of $A$. Since $A$ is a minimal finite skew brace, such that $A_{\oplus}$ is solvable, and $A_{\odot}$ is not solvable, and $H$ is a proper skew subbrace of $A$, the group $H_{\odot}$ is solvable. Let $P$ be a Sylow $p$-subgroup of $A_{\odot}$. This group is a $p$-group, therefore, $P$ is solvable. Since $|A_{\odot}:H_{\odot}|=p^{n}$ and $|A_{\odot}:P|$ are relatively prime, we have $A_{\odot}=H_{\odot}\odot P$. Therefore $A_{\odot}$ is a product of two solvable subgroups. Since $A_{\odot}$ is not divisible by $3$, from \cite[Lemma~2]{Sys} follows that $A_{\odot}$ is solvable, what contradicts the choice of $A$.
\end{proof}

Let $G$ be a finite group, and $p$ be a prime number. Recall that a subgroup $H$ of $G$ is said to be a $p$-complement if the index $|G:H|$ is a power of $p$, and the order $|H|$ is relatively prime with $p$. 
The following statement proved in \cite[Theorem 7]{Kazarin} states when a finite simple group has a $p$-complement. 
\begin{lemma}\label{composition factors}Let $G$ be a finite simple group, and $p$ be a prime divisor of $|G|$. If $G$ has a $p$-complement, then one of the following statements hold 
\begin{enumerate}
\item $G=A_p$;
\item $G={\rm PSL}_n(q)$, where $\frac{q^n-1}{q-1}$ is a power of $p$;
\item $G={\rm PSL}_2(11)$, $p=11$;
\item $G=M_{p}$, $p=11$ or  $p=23$.
\end{enumerate} 
\end{lemma}
Lemma~\ref{composition factors} has the following corollary. 
\begin{lemma}\label{solvablepcomp}Let $G$ be a finite group, and $p$ be a prime divisor of $|G|$. If $G$ has a solvable $p$-complement, then each non-abelian composition factor of $G$ is one of
\begin{enumerate}
\item ${\rm PSL}_2(q)$, where $q+1$ is a power of $p$;
\item ${\rm PSL}_2(7)$, where $p=7$;
\item ${\rm PSL}_3(3)$, where $p=13$.
\end{enumerate}
\end{lemma}
\begin{proof} Since $G$ has a solvable $p$-complement, each composition factor of $G$ also has a solvable $p$-complement. Since composition factors of $G$ are simple, from Lemma~\ref{composition factors} follows that each non-abelian composition factor is one of
\begin{enumerate}
\item $A_p$,
\item ${\rm PSL}_n(q)$, where $\frac{q^n-1}{q-1}$ is a power of $p$,
\item ${\rm PSL}_2(11)$, $p=11$,
\item $M_{p}$, $p=11$ or  $p=23$.
\end{enumerate} 
Let us check which cases from the list above are possible. It is clear that the $p$-complement in the alternating group $A_p$ is isomorphic to $A_{p-1}$. From the condition that the $p$-complement is solvable follows that $A_{p-1}$ is either $A_2$, or $A_3$, or $A_4$, so, $p\in \{3,4,5\}$. The group $A_3$ is abelian (but the statement of the lemma is about non-abelian composition factors), the group $A_4$ is not simple (but composition factors are simple), therefore the only possible case is that the composition factor is isomorphic to $A_5$, and $p=5$. Since  $A_5$ is isomorphic to ${\rm PSL}_2(4)$, and $(4^2-1)/(4-1)=5$, the composition factor is isomorphic to ${\rm PSL}_2(q)$, where $\frac{q^2-1}{q-1}=q+1$ is a power of $p$.

Suppose that the group ${\rm PSL}_2(11)$ has a solvable $11$-complement $A$. This complement is contained in some maximal subgroup $B$ of ${\rm PSL}_2(11)$. Since $|{\rm PSL}_2(11):A|$ is a power of $11$, and 
$$|{\rm PSL}_2(11):A|=|{\rm PSL}_2(11):B||B:A|,$$ 
the indices $|{\rm PSL}_2(11):B|$, $|B:A|$ are powers of $11$. Since $|{\rm PSL}_2(11):B|$ is a power of $11$, according to the atlas of finite groups \cite[Page 7]{Atlas} the group $B$ is isomorphic to $A_5$, and $|{\rm PSL}_2(11):B|=11$. Since $|B:A|$ is a power of $11$, and $|B|=|{\rm PSL}_2(11)|/11=60$ is not divisible by $11$, we have $A=B$, therefore $A$ is not solvable. This contradiction shows that the case ${\rm PSL}_2(11)$, $p=11$ is impossible.

Suppose that the group $M_{11}$ has a solvable $11$-complement $A$. This complement is contained in some maximal subgroup $B$ of $M_{11}$. Since $|M_{11}:A|$ is a power of $11$, and 
$$|M_{11}:A|=|M_{11}:B||B:A|,$$ 
the indices $|M_{11}:B|$, $|B:A|$ are powers of $11$. Since $|M_{11}:B|$ is a power of $11$, according to the atlas of finite groups \cite[Page 18]{Atlas} the group $B$ is isomorphic to $M_{10}$ which is the extension of $A_6$ by the graph automorphism, and $|M_{11}:B|=11$. Since $|B:A|$ is a power of $11$, and $|B|=|M_{11}|/11=2^33^25$ is not divisible by $11$, we have $A=B$, therefore $A$ is not solvable. This contradiction shows that the case $M_{11}$, $p=11$ is impossible.

Suppose that the group $M_{23}$ has a solvable $23$-complement $A$. This complement is contained in some maximal subgroup $B$ of $M_{23}$. Since $|M_{23}:A|$ is a power of $23$, and 
$$|M_{23}:A|=|M_{23}:B||B:A|,$$ 
the indices $|M_{23}:B|$, $|B:A|$ are powers of $23$. Since $|M_{23}:B|$ is a power of $23$, according to the atlas of finite groups \cite[Page 105]{Atlas} the group $B$ is isomorphic to $M_{22}$, and $|M_{23}:B|=23$. Since $|B:A|$ is a power of $23$, and $|B|=|M_{23}|/23=2^7\cdot3^2\cdot5\cdot7\cdot11$ is not divisible by $23$, we have $A=B$, therefore $A=M_{22}$ is not solvable. This contradiction shows that the case $M_{23}$, $p=23$ is impossible. 

Hence, the only possible composition factors of $G$ are the groups isomorphic to  ${\rm PSL}_n(q)$, where $(q^n-1)/(q-1)$ is a power of $p$. Let $A$ be a solvable $p$-complement in ${\rm PSL}_n(q)$. The group ${\rm PSL}_n(q)$ has a subgroup $B$ isomorphic to ${\rm PSL}_{n-1}(q)$. The order of $B$ is relatively prime with $(q^n-1)/(q-1)$, and since $(q^n-1)/(q-1)$ is a power of $p$, the order of $B$ is relatively prime with $p$. Therefore the subgroup $B$ of ${\rm PSL}_{n}(q)$ is contained in some $p$-complement $C$ of ${\rm PSL}_{n}(q)$. Since $A$ and $C$ are two $p$-complements in ${\rm PSL}_n(q)$, according to \cite[Proposition~1]{ButRev} there exists an automorphism $\varphi$ of ${\rm PSL}_n(q)$ such that $\varphi(A)=C$, in particular, $C$ is a solvable $p$-complement. Since $B$ is a subgroup of $C$, the group $B$ is solvable. Since $B$ is isomorphic to ${\rm PSL}_{n-1}(q)$, and ${\rm PSL}_{n-1}(q)$ is solvable only in the cases when ${\rm PSL}_{n-1}(q)={\rm PSL}_{1}(q)$, ${\rm PSL}_{n-1}(q)={\rm PSL}_{2}(2)$, or ${\rm PSL}_{n-1}(q)={\rm PSL}_{2}(3)$, the non-abelian composition factor of $G$ which is isomorphic to ${\rm PSL}_n(q)$, where $(q^n-1)/(q-1)$ is a power of $p$, is one of the following three
\begin{enumerate}
\item ${\rm PSL}_2(q)$, where $q+1$ is a power of $p$;
\item ${\rm PSL}_3(2)$, where $p=7$;
\item ${\rm PSL}_3(3)$, where $p=13$.
\end{enumerate}
Since ${\rm PSL}_3(2)$ is isomorphic to ${\rm PSL}_2(7)$, the lemma is proved.
\end{proof}
The following statement is contained in \cite[Table 5.3.A]{Kleid}.
\begin{lemma}\label{letsrepresent}Let $p,q$ be powers of different primes, and $q\neq 4,9$. If there exists a faithful representation ${\rm PSL}_2(q)\to {\rm GL}_n(p)$, then  $n\geq (q-1)/(2,q-1)$, where $(2,q-1)$ denotes the greatest commond divisor of $2$ and $q-1$.
\end{lemma}
\section{Proof of the main theorem}
In this section we prove the following main theorem of the paper.
\begin{theorem}\label{maintheorem}Let $A$ be a minimal finite skew brace, such that $A_{\oplus}$ is solvable, and $A_{\odot}$ is not solvable. Then $A_{\odot}$ is not simple.
\end{theorem}
\begin{proof} Since $A_{\oplus}$ is solvable, the index $|A_{\oplus}:[A_{\oplus},A_{\oplus}]_{\oplus}|$ is greater than $1$. Let $p$ be a  prime divisor of $|A_{\oplus}:[A_{\oplus},A_{\oplus}]_{\oplus}|$. From Lemma~\ref{characteristicsub} follows that there exists a characteristic subgroup $H$ of $A_{\oplus}$ such that $A_{\oplus}/H=\mathbb{Z}_p^n$ for some $n>0$. Since $H$ is characteristic subgroup of $A_{\oplus}$, for all $a\in A$ we have $\lambda_a(H)=H$, hence, the multiplicative group $A_{\odot}$ acts on $A_{\oplus}/H=\mathbb{Z}_p^n$. Therefore there is a homomorphism 
$$\varphi:A_{\odot}\to {\rm Aut}(A_{\oplus}/H)={\rm Aut}(\mathbb{Z}_p^n)={\rm GL}_n(p),$$ 
which maps an elememt $a\in A_{\odot}$ to the automorphism $\lambda_a$ of $A/H$. 

Suppose by contrary that $A_{\odot}$ is simple. In this case the kernel of $\varphi$ is either trivial, or coincides with $A_{\odot}$. Suppose that ${\rm ker}(\varphi)=A_{\odot}$. It means that for all $a,b\in A$ we have $\lambda_a(b\oplus H)=b\oplus H$, or 
$
\lambda_a(b)=b\oplus h
$
for some $h\in H$ ($h$ depends on $a,b$). Hence, from formula~(\ref{9onverse0}) follows that for every element $a\in A$ we have $a^{-1}=\ominus\lambda_{a}^{-1}(a)=\ominus(a\oplus h_1)$, where $h_1$ is an element from $H$ (which depends on $a$). Using these equalities, for arbitrary element $a\in A$, $b\in H$ we have
\begin{align*}
a^{-1}\odot b\odot a&=a^{-1}\oplus\lambda_a^{-1}(b\odot a)&&\\
&=\ominus(a\oplus h_1)\oplus (b\odot a)\oplus h_2&&h_1,h_2\in H\\
&=\ominus h_1\ominus a\oplus b\oplus \lambda_b(a)\oplus h_2&&h_1, h_2\in H\\
&=\ominus h_1\ominus a\oplus b\oplus a\oplus h_3\oplus h_2&&h_1,h_2,h_3\in H\\
&=\ominus h_1\oplus(\ominus a\oplus b\oplus a)\oplus h_3\oplus h_2&& h_1,h_2,h_3\in H,
\end{align*}
and since $\ominus a\oplus b\oplus a$ belongs to $H$, we conclude that $H$ is a normal subgroup of $A_{\odot}$. Since $|A_{\odot}:H|=p^n$ for $n>0$, the subgroup $H$ is a proper normal subgroup of $A_{\odot}$, what contradicts to the simplicity of $A_{\odot}$. Therefore, the case ${\rm ker}(\varphi)=A_{\odot}$ is impossible, and we have a faithful representation 
\begin{equation}\label{faithful}
\varphi:A_{\odot}\to {\rm GL}_n(p).
\end{equation}

Since $\lambda_a$ is an automorphism of $A_{\oplus}$, and $H$ is a characteristic subgroup of $A_{\oplus}$, for all $a,b\in H$ the element $\lambda_a(b)$ belongs to $H$, and the element $a\odot b=a\oplus\lambda_a(b)$ belongs to $H$. Therefore $H$ is a subgroup of $A_{\odot}$, and hence $H$ is a skew subbrace of $A$. Since $A$ is a minimal finite skew brace, such that $A_{\oplus}$ is solvable, and $A_{\odot}$ is not solvable, the multiplicative group $H_{\odot}$ is solvable, therefore by Hall theorem \cite[Theorem~20.1.1]{Osnovy} it contains a $p$-complement $P$ (which is solvable). Since $|A_{\odot}:H|=p^n$, and $P$ is a solvable $p$-complement in $H$, the group $P$ is a solvable $p$-complement in $A_{\odot}$. Since $A_{\odot}$ is simple, from Lemma~\ref{solvablepcomp} follows that one of the following cases hold
\begin{enumerate}
\item $A_{\odot}={\rm PSL}_2(q)$, where $q+1$ is a power of $p$;
\item $A_{\odot}={\rm PSL}_2(7)$, where $p=7$;
\item $A_{\odot}={\rm PSL}_3(3)$, where $p=13$.
\end{enumerate}
If we show that no of these cases possible, then we show that $A_{\odot}$ cannot be simple. We will prove that no of the three cases above possible considering the cases collected in the following table.
	\begin{center}
		\begin{tabular}{|l|l|l|}
			\hline
			1) $A_{\odot}={\rm PSL}_2(q)$, $q+1$ is a power of $p$~& 1.1) $q$ is odd&1.1.1) $q=3$ \\\cline{3-3}
			&In this case $p=2$ &1.1.2) $q=7$ \\\cline{3-3}
			& &1.1.3) $q>7$ \\
			\cline{2-3}
			& 1.2) $q$ is even&1.2.1) $q=2$ \\\cline{3-3}
			&In this case $p$ is odd &1.2.2) $q=4$ \\\cline{3-3}
			& &1.2.3) $q>4$ \\\cline{3-3}
			\hline
			\multicolumn{3}{|l|}{2) $A_{\odot}={\rm PSL}_2(7)$, $p=7$} \\
			\hline
			\multicolumn{3}{|l|}{3) $A_{\odot}={\rm PSL}_3(3)$, $p=13$} \\
			\hline
		\end{tabular}
	\end{center}

\noindent \textbf{Case 1:} $A_{\odot}={\rm PSL}_2(q)$, where $q+1$ is a power of $p$.

\noindent \textbf{Case 1.1:} $q$ is odd. In this case $q+1$ is even, and therefore $p=2$.

\noindent \textbf{Case 1.1.1:} $A_{\odot}={\rm PSL}_2(3)$. This case is impossible since the group $A_{\odot}={\rm PSL}_2(3)$ is solvable, but by the condition of the theorem the group $A_{\odot}$ is not solvable.

\noindent \textbf{Case 1.1.2:} $A_{\odot}={\rm PSL}_2(7)$. Since $A_{\oplus}/H=\mathbb{Z}_2^n$ and $|A_{\oplus}|=|A_{\odot}|=|{\rm PSL}_2(7)|=2^3\cdot 3\cdot 7$, the number $n$ is either $1$, or $2$, or $3$. From formula (\ref{faithful}) follows that there exists a faithful representation
$$A_{\odot}\to {\rm GL}_n(2).$$ 
Since $|{\rm GL}_1(2)|=1$, $|{\rm GL}_2(2)|=6$ and $|A_{\odot}|=|{\rm PSL}_2(7)|=168$, the cases $n=1,2$ are impossible, hence, $n=3$, and $|H|=3\cdot 7=21$.

Let $X$ be a Sylow $7$-subgroup of $H$. This group is a cyclic group of order $7$. By Sylow theorem, $X$ is a unique Sylow $7$-subgroup of $H$, therefore $X$ is a characteristic subgroup of $H$. Since $H$ is a characteristic subgroup of $A_{\oplus}$, and $X$ is a characteristic subgroup of $H$, the group $X$ is a characteristic subgroup of $A_{\oplus}$. Denote by $G=A_{\oplus}\rtimes A_{\odot}=A_{\oplus}\rtimes {\rm PSL}_2(7)$, where $A_{\odot}={\rm PSL}_2(7)$ acts on $A_{\oplus}$ by automorphisms $\lambda_a$ for $a\in A_{\odot}$.

Since $X$ is characteristic in $A_{\oplus}$, and  $A_{\odot}={\rm PSL}_2(7)$ acts on $A_{\oplus}$, the group $A_{\odot}={\rm PSL}_2(7)$ acts on $X$, therefore $X$ is a normal subgroup of $G$. Let $C=C_G(X)$ be the centralizer of $X$ in $G$. Since $X$ is a normal subgroup of $G$, the group $C$ is also a normal subgroup of $G$. Since $X$ is cyclic, the group ${\rm Aut}(X)$ is abelian, and since $A_{\odot}={\rm PSL}_2(7)$ is simple, the group $A_{\odot}={\rm PSL}_2(7)$ acts trivially on $X$, so, $A_{\odot}={\rm PSL}_2(7)$ is a subgroup of $C$.

Since $C$ is a normal subgroup of $G$, the group $CH/H$ is a normal subgroup of 
$$G/H=(A_{\oplus}/H)\rtimes A_{\odot}=(\mathbb{Z}_2)^3\rtimes {\rm PSL}_2(7).$$
Moreover, since $A_{\odot}={\rm PSL}_2(7)$ is a subgroup of $C$, we have the following inclusions
$${\rm PSL}_2(7)\leq CH/H\leq(\mathbb{Z}_2)^3\rtimes {\rm PSL}_2(7).$$
From this inclusions follows that $CH/H=B\rtimes {\rm PSL}_2(7)$, where $B$ is a subgroup of $\mathbb{Z}_2^3$. If $B$ is a proper subgroup of $\mathbb{Z}_2^3$, then the semidirect product 
$$(A_{\oplus}/H)\rtimes A_{\odot}=(\mathbb{Z}_2)^3\rtimes {\rm PSL}_2(7)$$
is a direct product (since ${\rm PSL}_2(7)$ cannot act nontrivially on $B<\mathbb{Z}_2^3$). Therefore, the homomorphism 
$$\varphi:A_{\odot}\to {\rm Aut}(A_{\oplus}/H)={\rm GL}_3(2),$$ 
which maps an elememt $a\in A_{\odot}$ to the automorphism $\lambda_a$ of $A/H$ is trivial, what contradicts formula~(\ref{faithful}). Hence $B$ cannot be a proper subgroup of $B$, i.~e. either $CH/H={\rm PSL}_2(7)$ or $CH/H=(\mathbb{Z}_2)^3\rtimes {\rm PSL}_2(7)$. Let us prove that no of these cases is possible.

If $CH/H={\rm PSL}_2(7)$, then due to the fact that $CH/H$ is normal in $G/H=(\mathbb{Z}_2)^3\rtimes {\rm PSL}_2(7)$, the semidirect product $(A_{\oplus}/H)\rtimes A_{\odot}=(\mathbb{Z}_2)^3\rtimes {\rm PSL}_2(7)$
is a direct product. Therefore, the homomorphism 
$\varphi:A_{\odot}\to {\rm Aut}(A_{\oplus}/H)={\rm GL}_3(2)$, 
which maps an elememt $a\in A_{\odot}$ to the automorphism $\lambda_a$ of $A/H$ is trivial, what contradicts formula~(\ref{faithful}). Hence, the case  $CH/H={\rm PSL}_2(7)$ is impossible, and we have $CH/H=G/H$, or equivalently $CH=G$. Here we have two cases: either $C\cap A_{\oplus}=A_{\oplus}$ or $C\cap A_{\oplus}\neq A_{\oplus}$. 

If $C\cap A_{\oplus}=A_{\oplus}$, then $A_{\oplus}$ normalizes $X$, therefore $A_{\oplus}=X\times D$, where $D$ is a Hall $\{2,3\}$-subgroup of $A_{\oplus}$. Since $A_{\oplus}=X\times D$, where $D$ is a Hall $\{2,3\}$-subgroup of $A_{\oplus}$, the group  $D$ is characteristic in $A_{\oplus}$. Therefore the group ${\rm PSL}_2(7)$ acts on $A_{\oplus}/D$. Since the group $A_{\oplus}/D$ has order $|A_{\oplus}/D|=|A_{\oplus}|/|D|=|X|=7$, it is a cyclic group, and $|{\rm Aut}(A_{\oplus}/D)|$ is abelian. Since $A_{\odot}={\rm PSL}_2(7)$ is simple, the group $A_{\odot}={\rm PSL}_2(7)$ acts trivially on $A_{\oplus}/D$.  It means that for all $a,b\in A$ we have $\lambda_a(b\oplus D)=b\oplus D$, or 
$
\lambda_a(b)=b\oplus d
$
for some $d\in D$ ($d$ depends on $a,b$). 
Hence, from formula~(\ref{9onverse0}) follows that for every element $a\in A$ we have $a^{-1}=\ominus\lambda_{a}^{-1}(a)=\ominus(a\oplus d_1)$, where $d_1$ is an element from $D$ (which depends on $a$). Using these equalities, for arbitrary element $a\in A$, $b\in D$ we have
\begin{align*}
a^{-1}\odot b\odot a&=a^{-1}\oplus\lambda_a^{-1}(b\odot a)&&\\
&=\ominus(a\oplus d_1)\oplus (b\odot a)\oplus d_2&&d_1,d_2\in D\\
&=\ominus d_1\ominus a\oplus b\oplus \lambda_b(a)\oplus d_2&&d_1, d_2\in D\\
&=\ominus d_1\ominus a\oplus b\oplus a\oplus d_3\oplus d_2&&d_1,d_2,d_3\in D\\
&=\ominus d_1\oplus(\ominus a\oplus b\oplus a)\oplus d_3\oplus d_2&& d_1,d_2,d_3\in D,
\end{align*}
and since $\ominus a\oplus b\oplus a$ belongs to $D$, we conclude that $D$ is a normal subgroup of $A_{\odot}$. Since $|A_{\odot}:D|=7$, the subgroup $D$ is a proper normal subgroup of $A_{\odot}$, what contradicts to the simplicity of $A_{\odot}$. So, the case $C\cap A_{\oplus}=A_{\oplus}$ is impossible.

If $C\cap A_{\oplus}\neq A_{\oplus}$, then since $CH=G$, the group $C$ contains a Sylow $2$-subgroup of $A_{\oplus}$. Since $C\cap A_{\oplus}\neq A_{\oplus}$, and $C\cap A_{\oplus}$ contains both Sylow $7$-subgroup $X$ of $A_{\oplus}$ and Sylow $2$-subgroup of $A_{\oplus}$, the group $C\cap A_{\oplus}$ is a Hall $\{2,7\}$-subgroup of $A_{\oplus}$.  Since $C,A_{\oplus}$ are normal subgroups of $G$, the group $C\cap A_{\oplus}$ is a normal subgroup of $G$. Therefore, the group $A_{\odot}$ acts on $C\cap A_{\oplus}$, and hence $C\cap A_{\oplus}$ is a skew subbrace of $A$, in particular, $C\cap A_{\oplus}$ is a Hall $\{2,7\}$-subgroup of $A_{\odot}$. So, the group $A_{\odot}$ has a  $2$-complement $H$, and a $3$-complement $C\cap A_{\oplus}$, therefore by \cite[Corollary~4.4]{AraWar} the group $A_{\odot}$ is solvable, what contradicts the choice of $A$.

\noindent \textbf{Case 1.1.3:} $A_{\odot}={\rm PSL}_2(q)$, where $q>7$. The order of the group $A_{\odot}={\rm PSL}_2(q)$ is equal to
$$|{\rm PSL}_2(q)|=\frac{(q-1)q(q+1)}{2}.$$
Therefore the maximal power of $2$ which divides $|{\rm PSL}_2(q)|$ is equal to $(q+1)$ ($(q+1)$ is a power of~$2$, $q$ is odd, and $q-1$ is divisible by 2, but not divisible by $4$). From formula (\ref{faithful}) follows that there exists a faithful representation
$${\rm PSL}_2(q)\to {\rm GL}_n(2),$$ 
where $n\leq {\rm log}_2(q+1)$. From the other side from Lemma~\ref{letsrepresent} follows that if there exists a faithful representation ${\rm PSL}_2(q)\to {\rm GL}_n(2)$, then $n\geq (q-1)/2$. From these two inequalities follows that ${\rm log}_2(q+1)\geq (q-1)/2$ or $q+1\geq 2^{(q-1)/2}$. This inequality is impossible for $q>7$, therefore the case $A_{\odot}={\rm PSL}_2(q)$, where $q+1$ is a power of $p$, and $q>7$ is impossible.

\noindent\textbf{Case 1.2:} $q$ is even.

\noindent \textbf{Case 1.2.1:} $A_{\odot}={\rm PSL}_2(2)$. This case is impossible since the group $A_{\odot}={\rm PSL}_2(2)$ is solvable, but by the condition of the theorem the group $A_{\odot}$ is not solvable.

\noindent \textbf{Case 1.2.2:} $A_{\odot}={\rm PSL}_2(4)$. In this case $p=5$. Since the order of $A_{\odot}={\rm PSL}_2(4)$ is equal to $|{\rm PSL}_2(4)|=2^2\cdot 3\cdot 5$, and $|A_{\oplus}:H|=p^n=5^n$, we conclude that $n=1$. From formula~(\ref{faithful}) follows that there exists a faithful representation 
$$
\varphi:{\rm PSL}_2(4)\to {\rm GL}_1(5).
$$
But such representation cannot exist due to the equalities $|{\rm PSL}_2(4)|=2^2\cdot 3\cdot 5$, ${\rm GL}_1(5)=4$. Hence, the case $A_{\odot}={\rm PSL}_2(4)$, $p=5$ is impossible.

\noindent \textbf{Case 1.2.3:} $A_{\odot}={\rm PSL}_2(q)$ for $q>4$ is impossible. Since $q$ is even, $p$ is odd, hence $p\geq 3$. The order of the group $A_{\odot}={\rm PSL}_2(q)$ is equal to
$$|{\rm PSL}_2(q)|=(q-1)q(q+1).$$
Therefore the maximal power of $p$ which divides $|{\rm PSL}_2(q)|$ is equal to $(q+1)$ ($(q+1)$ is a power of $p$, and $q(q-1)$ is relatively prime with $p$). From formula (\ref{faithful}) follows that there exists a faithful representation
$${\rm PSL}_2(q)\to {\rm GL}_n(p),$$ 
where $n\leq {\rm log}_p(q+1)$. From the other side from Lemma~\ref{letsrepresent} follows that if there exists a faithful representation ${\rm PSL}_2(q)\to {\rm GL}_n(p)$, then $n\geq (q-1)$. From these two equalities follows that 
$${\rm log}_p(q+1)\geq (q-1)$$ 
or $(q+1)\geq p^{q-1}\geq 3^{q-1}$. This inequality is impossible for $q>4$, therefore the case $A_{\odot}={\rm PSL}_2(q)$, where $q+1$ is a power of $p$, and $q>4$ is impossible.

\noindent \textbf{Case 2:} $A_{\odot}={\rm PSL}_3(3)$, $p=13$. Since the order of $A_{\odot}={\rm PSL}_3(3)$ is equal to $|{\rm PSL}_3(3)|=2^4\cdot 3\cdot 13$, and $|A_{\oplus}:H|=p^n=13^n$, we conclude that $n=1$. From formula~(\ref{faithful}) follows that there exists a faithful representation 
$$
\varphi:{\rm PSL}_3(3)\to {\rm GL}_1(13).
$$
But such representation cannot exist due to the equalities $|{\rm PSL}_3(3)|=2^4\cdot 3\cdot 13$, ${\rm GL}_1(13)=12$. Hence, the case $A_{\odot}={\rm PSL}_3(3)$, $p=13$ is impossible.

\noindent \textbf{Case 3:} $A_{\odot}={\rm PSL}_2(7)$, $p=7$. Since the order of $A_{\odot}={\rm PSL}_2(7)$ is equal to $|{\rm PSL}_2(7)|=2^3\cdot 3\cdot 7$, and $|A_{\oplus}:H|=p^n=7^n$, we conclude that $n=1$. From formula~(\ref{faithful}) follows that there exists a faithful representation 
$$
\varphi:{\rm PSL}_2(7)\to {\rm GL}_1(7).
$$
But such representation cannot exist due to the equalities $|{\rm PSL}_2(7)|=2^3\cdot 3\cdot 7$, ${\rm GL}_1(7)=6$. Hence, the case $A_{\odot}={\rm PSL}_2(7)$, $p=7$ is impossible.
\end{proof}

{\small

\medskip

\medskip

\noindent
Ilya Gorshkov\\
Sobolev Institute of Mathematics, Acad. Koptyug avenue 4, 630090 Novosibirsk, Russia\\
ilygor8@gmail.com

~\\
Timur Nasybullov\\
Sobolev Institute of Mathematics, Acad. Koptyug avenue 4, 630090 Novosibirsk, Russia\\
Novosibirsk State University, Pirogova 1, 630090 Novosibirsk, Russia\\
ntr@math.nsc.ru

}
\end{document}